\documentclass{amsart}
\usepackage{amssymb}
\usepackage{amsmath}
\usepackage{bm}				% Bold Math
\usepackage{mathtools}		% \begin{bmatrix*}[r] for right-align
\usepackage{mathrsfs} 		% mathscr
\usepackage{geometry}
\usepackage{setspace}
\usepackage{hyperref}

\usepackage{mathrsfs}

\newcommand{\Nbb}{\mathbb N}

\newcommand{\Qbb}{\mathbb Q}
\newcommand{\Rbb}{\mathbb R}

\newcommand{\Fcal}{\mathcal F}

\newcommand{\Kcal}{\mathcal K}

\newcommand{\Mcal}{\mathcal M}

\newcommand{\Pcal}{\mathcal P}

\newcommand{\Tcal}{\mathcal T}
\newcommand{\Ucal}{\mathcal U}

\newcommand{\Cscr}{\mathscr C}

\newcommand{\Pscr}{\mathscr P}

\renewcommand{\to}{\rightarrow}

\newcommand{\normal}{\vartriangleleft}
\newcommand{\dd}{\:\mathrm{d}}
\DeclareMathOperator{\cl}{c\ell}
\DeclareMathOperator{\supp}{supp}

\DeclareMathOperator{\conv}{conv}
\DeclareMathOperator*{\plim}{\mathit{p}-lim}
\DeclareMathOperator*{\qlim}{\mathit{q}-lim}
\newcommand{\Linfty}{L_\infty(G)}
\newcommand{\LinftyS}{L_\infty^*(G)}
\newcommand{\Lone}{L_1(G)}
\newcommand{\Ltwo}{L_2(G)}
\newcommand{\VN}{VN(G)}
\newcommand{\AG}{A(G)}
\newcommand{\Ghat}{\widehat{G}}
\newcommand{\cf}{{\it c.f.}}

\DeclareMathOperator{\TLIM}{TLIM}
\DeclareMathOperator{\TRIM}{TRIM}
\DeclareMathOperator{\TIM}{TIM}
\DeclareMathOperator{\FC}{FC}

\theoremstyle{definition}
\newtheorem{Theorem}{Theorem}[section]
\newtheorem{Lemma}[Theorem]{Lemma}
\newtheorem{Corollary}[Theorem]{Corollary}

\newtheorem{point}[Theorem]{}

\makeatletter

\def\chaptermark#1{}%whatever

\def\chapter{%
  \thispagestyle{plain}\global\@topnum\z@
  \@afterindenttrue \secdef\@chapter\@schapter}

\def\@chapter[#1]#2{\refstepcounter{chapter}%
  \ifnum\c@secnumdepth<\z@ \let\@secnumber\@empty
  \else \let\@secnumber\thechapter \fi
  \typeout{\chaptername\space\@secnumber}%
  \def\@toclevel{0}%
  \ifx\chaptername\appendixname \@tocwriteb\tocappendix{chapter}{#2}%
  \else \@tocwriteb\tocchapter{chapter}{#2}\fi
  \chaptermark{#1}%
  \addtocontents{lof}{\protect\addvspace{10\p@}}%
  \addtocontents{lot}{\protect\addvspace{10\p@}}%
  \@makechapterhead{#2}\@afterheading}
\def\@schapter#1{\typeout{#1}%
  \let\@secnumber\@empty
  \def\@toclevel{0}%
  \ifx\chaptername\appendixname \@tocwriteb\tocappendix{chapter}{#1}%
  \else \@tocwriteb\tocchapter{chapter}{#1}\fi
  \chaptermark{#1}%
  \addtocontents{lof}{\protect\addvspace{10\p@}}%
  \addtocontents{lot}{\protect\addvspace{10\p@}}%
  \@makeschapterhead{#1}\@afterheading}
\newcommand\chaptername{Chapter}

\def\@makechapterhead#1{\global\topskip 7.5pc\relax
  \begingroup
  \fontsize{\@xivpt}{18}\bfseries\centering
     \Roman{chapter}. #1\par \endgroup
  \skip@22\p@ \advance\skip@-\normalbaselineskip
  \vskip\skip@ }
\def\@makeschapterhead#1{\global\topskip 7.5pc\relax
  \begingroup
  \fontsize{\@xivpt}{18}\bfseries\centering
  #1\par \endgroup
  \skip@34\p@ \advance\skip@-\normalbaselineskip
  \vskip\skip@ }
\def\appendix{\par
  \c@chapter\z@ \c@section\z@
  \let\chaptername\appendixname
  \def\thechapter{\@Alph\c@chapter}}

\newcounter{chapter}

\newif\if@openright

\makeatother

\begin{document}
\title{Counting topologically invariant means on $L_\infty(G)$ and $VN(G)$ with ultrafilters}
\author{John Hopfensperger}
\thanks{The present paper will form part of the author's PhD thesis under supervision of Ching Chou.}
% \date{October 2019}
\address{\hskip-\parindent
Department of Mathematics, University at Buffalo,
Buffalo, NY 14260-2900, USA}
\email{johnhopf@buffalo.edu}
% \subjclass[2010]{43A07, 43A30, 20F24}
\begin{abstract}
In 1970, Chou showed there are $|\Nbb^*| = 2^{2^\mathbb{N}}$ topologically invariant means on $L_\infty(G)$ for any noncompact, $\sigma$-compact amenable group.
Over the following 25 years, the sizes of the sets of topologically invariant means on $L_\infty(G)$ and $VN(G)$ were determined for any locally compact group.
Each paper on a new case reached the same conclusion -- ``the cardinality is as large as possible'' -- but a unified proof never emerged.
In this paper, I show $L_1(G)$ and $\AG$ always contain orthogonal nets converging to invariance.
An orthogonal net indexed by $\Gamma$ has $|\Gamma^*|$ accumulation points, where $|\Gamma^*|$ is determined by ultrafilter theory.

Among a smattering of other results, I prove Paterson's conjecture that left and right topologically invariant means on $L_\infty(G)$ coincide iff $G$ has precompact conjugacy classes.
\end{abstract}
\maketitle
%%%%%%%%%%%%%%%%%%%%%
%%%%%%%%%%%%%%%%%%%%%
%%%%%%%%%%%%%%%%%%%%%
\chapter{Background}
\section{History}
\begin{point}
F\o{}lner's condition for amenable discrete groups says,
for all finite $K\subset G$ and $\epsilon > 0$ there exists finite $F_{(K,\epsilon)}\subset G$ which is $(K,\epsilon)$-invariant.
The set $\Gamma$ of all ordered pairs $\gamma = (K, \epsilon)$ is a directed set, ordered by increasing $K$ and decreasing $\epsilon$.
Define $m_\gamma\in\ell_\infty^*(G)$ by $m_\gamma(f) = \frac{1}{|F_\gamma|} \sum_{x\in F_\gamma} f(x)$.
Then the net $\{m_\gamma\}_{\gamma\in\Gamma}$ converges to invariance, and any limit point is an invariant mean on $\ell_\infty(G)$.

The analogue of F\o{}lner's condition for locally compact amenable groups says,
for all compact $K\subset G$ and $\epsilon > 0$, there exists compact $F\subset G$ which is $(K,\epsilon)$-invariant.
The analogous definition of $m_\gamma\in\LinftyS$ is $m_\gamma(f) = \frac{1}{|F_\gamma|} \int_{F_\gamma} f(x) \dd{x}$.
If $m$ is any limit point of $\{m_\gamma\}_{\gamma\in\Gamma}$, it is not only invariant but topologically invariant.
That is, $m(\phi) = m(f*\phi)$ for any $\phi\in\Linfty$ and $f\in \Lone$ with $\|f\|_1 = \int_G f = 1$.
\end{point}

\begin{point}
Non-topologically invariant means are harder to come by.
All invariant means on a discrete group are topologically invariant, and
it was not until 1972 that \cite{Rudin72} and \cite{Granirer} independently discovered a construction of non-topologically invariant means for any non-discrete $G$ that is amenable-as-discrete.
Several papers have discussed the implications of this construction.
Notably, \cite{Rosenblatt76} combined it with a Baire category argument to construct $2^{2^\Nbb}$ invariant means that are not topologically invariant.

The general problem of enumerating non-topologically invariant means seems intractable.
For instance, the famously difficult Banach-Ruziewicz problem boils down to whether $SO(n,\Rbb)$ admits any non-topologically invariant means, \cf\ \cite[Proposition 1.3]{AllStars}.
\end{point}

\begin{point}
Let $\TLIM(G)$ denote the topologically left-invariant means on $\Linfty$, and $\TIM(G)$ the (two-sided) topologically invariant means.
In \cite{patersonFCBar}, Paterson conjectured that $\TLIM(G) = \TIM(G)$ iff $G$ has precompact conjugacy classes, and proved it assuming $G$ is compactly generated.
The short and insightful paper \cite{milnes} proved Paterson's conjecture assuming $G$ is $\sigma$-compact, and gave me the ideas to prove it in full generality.
\end{point}

\begin{point}
For any set $S$, let $|S|$ denote its cardinal number -- that is, the first ordinal $\alpha$ such that there exists a bijection from $\alpha$ to $S$.
(There exists an ordinal of each cardinality by the axiom of choice.)
By definition, an ordinal is the set of all previous ordinals.
Thus the cardinal usually called $\aleph_0$ is none other than $\Nbb = \{0,1,\hdots\}$.
\end{point}

\begin{point}\label{Intro_Number_of_Tims}
Let $\kappa = \kappa(G)$ be the first ordinal such that there is a family $\Kcal$ of compact subsets of $G$ with $|\Kcal| = \kappa$ and $G = \bigcup \Kcal$.
It's not hard to prove $|\TIM(G)| \leq |\TLIM(G)| \leq 2^{2^\kappa}$.
Of course when $\kappa = 1$, the unique topologically invariant mean is Haar measure.
But when $\kappa \geq \Nbb$, $|\TIM(G)|$ actually equals $2^{2^\kappa}$.
Here is an abbreviated history of this surprising result, which took almost 20 years to establish:

When $\kappa=\Nbb$, \cite{Chou70} defined $\pi: \Linfty\to\ell_\infty(\Nbb)$ by $\pi(f)(n) = \frac{1}{|U_n|}\int_{U_n} f$, where $\{U_n\}$ is a F\o{}lner sequence of mutually disjoint sets.
Thus $\pi^* : \ell_\infty^*(\Nbb) \to \LinftyS$ is an embedding.
Let $c_0 = \{f\in \ell_\infty(\Nbb) : \lim_n f(n) = 0\}$ and $\Fcal = \{m\in\ell_\infty^*(\Nbb) : \|m\| = 1, m \geq 0, m|_{c_0} \equiv 0\}$, so that $\pi^*[\Fcal]\subset\TLIM(G)$.
Regarding the nonprincipal ultrafilters on $\Nbb$ as elements of $\Fcal$, we see $|\TLIM(G)| \geq |\beta\Nbb - \Nbb| = 2^{2^\Nbb}$.

When $G$ is discrete and $\kappa \geq \Nbb$, \cite{Chou76} proved $|\TIM(G)| \geq 2^{2^\kappa}$.
When $G$ is non-discrete and $\kappa \geq \Nbb$, \cite{lau-paterson} proved $|\TLIM(G)| \geq 2^{2^\kappa}$.
These papers take the more direct approach of constructing disjoint, translation-invariant subsets of $\beta G$.

The full result $|\TIM(G)| \geq 2^{2^\kappa}$ was finally proved by \cite{Yang}.
Yang realized that the trick to generalizing Chou's embedding argument is to replace $\Nbb$ by the indexing set of a F\o{}lner net.
\end{point}

\begin{point}
When $G$ is discrete abelian, $\kappa(\Ghat) = 1$ and Haar measure is the unique topologically invariant mean on $L_\infty(\Ghat)$.
More generally, let $\mu = \mu(G)$ be the first ordinal such that $G$ has a neighborhood basis $\Ucal$ at the origin with $|\Ucal| = \mu$.
When $G$ is abelian, $\kappa(\Ghat) = \mu(G)$ by \cite[(24.48)]{HR1}.
Thus when $G$ is non-discrete abelian, $\kappa(\Ghat) \geq \Nbb$ and $|\TIM(\Ghat)| = 2^{2^\mu}$.

When $G$ is non-abelian, the group von Neumann algebra $\VN$ is the natural analogue of $L_\infty(\Ghat)$.
\\If $\TIM(\Ghat)$ denotes the set of topologically invariant means on $\VN$, then the analogous results hold:
\\When $\mu = 1$, $\TIM(\Ghat)$ is the singleton comprising the point-measure $\delta_e$, as proved by \cite{renaud}.
\\When $\mu \geq \Nbb$, $|\TIM(\Ghat)| = 2^{2^\mu}$.
This is proved by \cite{Chou82} when $\mu = \Nbb$, using an embedding $\pi^*: \ell_\infty^*(\Nbb) \to \VN^*$,
and by \cite{Hu95} when $\mu > \Nbb$, using a family $\{\pi_\gamma^*: \ell_\infty^*(\mu) \to \VN^*\}_{\gamma<\mu}$ of embeddings!
\end{point}

% \begin{point}\label{Definition of Gamma*}
% Once again, consider a net $\{m_\gamma\}_{\gamma \in \Gamma}$ converging to invariance.
% % This net could be in $\LinftyS$ or $\VN^*$, it doesn't matter.
% To say $m$ is a limit point of $\{m_\gamma\}_{\gamma\in\Gamma}$ is to say $m$ is in the closure of each tail $\{m_\gamma\}_{\gamma > \alpha}$.
% Therefore we define $\Gamma^*$ to be the set of all ultrafilters in $\beta\Gamma$ that include every tail $T_\alpha = \{\gamma\in\Gamma:\gamma > \alpha\}$.
% Now if $p\in \Gamma^*$, then $\plim_\gamma m_\gamma$ is an invariant mean.
% Furthermore, if the elements of $\{m_\gamma\}_{\gamma\in\Gamma}$ are orthogonal,
% then the map $p\mapsto \plim_\gamma m_\gamma$ is one-to-one, proving there are $|\Gamma^*|$ invariant means.

% Theorems about general sets of ultrafilters such as $\Gamma^*$ were not widely accessible before Hindman and Strauss published the first edition of their text \cite{hindman} in 1998, postdating any of the aforementioned papers.
% \end{point}
%%%%%%%%%%%%%%%%%%%%%
%%%%%%%%%%%%%%%%%%%%%
%%%%%%%%%%%%%%%%%%%%%
\section{Ultrafilters}
The following exposition is sparse. For more details, see \cite[Chapter 3]{hindman}.

\begin{point}
Suppose $\{x_\gamma\} = \{x_\gamma\}_{\gamma\in\Gamma}$ is an infinite subset of the compact Hausdorff space $X$.
Obviously $\{x_\gamma\}$ has limit points in $X$, but how do we ``name'' them?
Regarding $\Gamma$ as a discrete topological space, let $f: \Gamma \to X$ be the continuous function $\gamma \mapsto x_\gamma$.
Let $\tilde{f}: \beta\Gamma \to X$ be the unique continuous extension of $f$ to the Stone-\v{C}ech compactification of $\Gamma$.
Since $\tilde{f}[\beta\Gamma]$ is compact and has $f[\Gamma]$ as a dense subset,
the limit points of $\{x_\gamma\}$ are precisely the points $\tilde{f}(p)$ with $p\in \beta\Gamma$.
We usually write $\plim_\gamma x_\gamma$ instead of $\tilde{f}(p)$.
\end{point}

\begin{point}
In functional analysis, the Stone-\v{C}ech compactification $\beta\Gamma$ of a completely regular space $\Gamma$ is realized as the Gelfand spectrum of the Banach algebra $C(\Gamma)$.
In particular, suppose $\Gamma$ is discrete, and $p\in \beta\Gamma$ is a nonzero multiplicative functional on $C(\Gamma) = \ell_\infty(\Gamma)$.
For any $S\subset\Gamma$, we have $\langle p, 1_S\rangle = \langle p, 1_S\cdot 1_S\rangle = \langle p, 1_S\rangle^2$.
Thus $p$ may be regarded as a $\{0,1\}$-valued measure on $\Gamma$.

But, by definition, an ultrafilter on $\Gamma$ is nothing more than a set of the form $\{S\subset \Gamma : p(S) = 1\}$ where $p$ is a nonzero $\{0,1\}$-valued measure.
(Equivalently, it is a maximal collection of subsets of $\Gamma$ that is closed under finite intersections and does not contain $\varnothing$.)
In this way, the combinatoric construction of $\beta\Gamma$ as the set of all ultrafilters on $\Gamma$ is identical to the Gelfand spectrum construction.
\end{point}

\begin{point}
In functional analysis, we embed $\Gamma$ in $\beta\Gamma$ by sending $\gamma$ to the point-measure $\delta_\gamma$, which corresponds to the principal ultrafilter $p_\gamma = \{S\subset\Gamma : \gamma\in S\}$.
The defining feature of a principal ultrafilter is that its smallest element has cardinality 1.
For any cardinal $\kappa$, we can define the $\kappa$-uniform ultrafilters as those whose smallest elements have cardinality $\kappa$.
As we shall see, there are $2^{2^{|\Gamma|}}$ $|\Gamma|$-uniform ultrafilters on $\Gamma$.
Of course, since $\beta\Gamma\subset \Pscr(\Pscr(\Gamma))$, there can't be more than $2^{2^{|\Gamma|}}$ ultrafilters in total.
\end{point}

\begin{point}
Consider the case when $\Gamma$ is a directed set.
% As in \ref{Definition of Gamma*},
Let $\Gamma^*$ denote the set of ultrafilters in $\beta\Gamma$ that include every tail $T_\alpha = \{\gamma\in\Gamma : \gamma > \alpha\}$,
and suppose $|T_\alpha| = |\Gamma| = \kappa$ for each $\alpha \in \Gamma$.
Notice that any finite intersection $T_{\gamma_1} \cap \hdots \cap T_{\gamma_n}$ contains $T_\alpha$, where $\alpha \geq \gamma_1, \hdots, \gamma_n$, hence $|T_{\gamma_1} \cap \hdots \cap T_{\gamma_n}| = \kappa$.
\end{point}

\begin{Lemma} \label{Lemma Cardinality of Gamma*}
$\Gamma^*$ has cardinality $2^{2^\kappa}$, the same as $\beta\Gamma$.
\end{Lemma}
\begin{proof}
This follows immediately from \cite[Theorem 3.62]{hindman}, which actually says something a bit stronger:
\\Let $\Gamma$ be an infinite set with cardinality $\kappa$, and let $\Tcal$ be a collection of at most $\kappa$ subsets of $\Gamma$ such that $|\bigcap F| = \kappa$ for any finite $F \subset \Tcal$.
Then there are $2^{2^\kappa}$ $\kappa$-uniform ultrafilters containing $\Tcal$.
\end{proof}
%%%%%%%%%%%%%%%%%%%%%
%%%%%%%%%%%%%%%%%%%%%
%%%%%%%%%%%%%%%%%%%%%
\section{Projections and Means}
\begin{point}\label{Definition of Support}
A positive unital functional on a von Neumann algebra $X$ is called a state.
A state $u$ is called normal if $\langle u, \sup_\alpha P_\alpha\rangle = \sup_\alpha \langle u, P_\alpha\rangle$ for any family $\{P_\alpha\}\subset X$ of (orthogonal) projections.
Let $\Pcal_1$ denote the set of normal states on $X$.
Sakai famously proved that the linear span of $\Pcal_1$ forms the predual of $X$ -- that is, $\mathrm{span}(\Pcal_1)^* = X$.
For each $u\in \Pcal_1$, we can define a projection $S(u)$ called the support of $u$, which is the inf of all projections $P$ such that $\langle P, u\rangle = \langle I, u\rangle = 1$.
Now $u,v\in \Pcal_1$ are called orthogonal if $S(u) S(v) = 0$.
\end{point}

\begin{Lemma}\label{Lemma Orthogonal Implies Injection}
Suppose $\{u_\gamma\}_{\gamma\in\Gamma} \subset \Pcal_1$ are mutually orthogonal.
In other words, $S(u_\gamma) S(u_\beta) = 0$ when $\gamma\neq\beta$.
Then the map $\beta\Gamma \to X^*$ given by $p\mapsto \plim_\gamma u_\gamma$ is one-to-one.
\end{Lemma}
\begin{proof}
Suppose $p,q$ are distinct ultrafilters, say $E\in p$ and $E^C \in q$.
Let $P = \sup_{\gamma\in E} S(u_\gamma)$.
Clearly $\langle \plim_\gamma u_\gamma, P\rangle = 1$.
For any $\beta\in E^C$, $P \leq 1 - S(u_\beta)$, hence $\langle u_\beta, P\rangle = 0$, hence $\langle\qlim_\gamma u_\gamma, P\rangle = 0$.
\end{proof}

\begin{point}\label{Linfty is operators}
For example, $\Linfty$ is a von Neumann algebra of multipliers on $\Ltwo$, and its predual is $\Lone$.
Let $\Pcal_1(G)$ denote the normal states on $\Linfty$.
By Sakai's result, the linear span of $\Pcal_1(G)$ is $\Lone$, 
so $\Pcal_1(G)$ itself must be $\{f\in \Lone : f\geq 0, \|f\|_1 = 1\}$.
Given $f\in\Pcal_1(G)$, let $\supp(f) = \{x\in G: f(x)\neq 0\}$.
Clearly $S(f)$ is the indicator function $1_{\supp(f)}$.
From this, we conclude $f,g \in \Pcal_1(G)$ are orthogonal if $\supp(f)\cap \supp(g) = \varnothing$.
\end{point}

\begin{point}
Let $X$ be any von Neumann algebra.
Endow $X^*$ with the $w^*$-topology,
and let $\Mcal\subset X^*$ be the set of all states.
In the context of amenability, $\Mcal$ is traditionally called the set of means on $X$.
Notice $\|m\|=1$ for each $m\in\Mcal$, since $T \leq \|T\| I$ for any self-adjoint $T$.
Notice $\Pcal_1$ is convex and $\Mcal$ is compact convex.
\end{point}

\begin{Lemma}[Hahn-Banach]\label{Lemma Hahn Banach}
Suppose $A, B\subset X^*$ are disjoint compact convex sets.
Then they are separated by some $T\in X$, in the sense
	$\inf_{a\in A}\Re\langle a, T\rangle > \sup_{b\in B}\Re\langle b, T\rangle$.
If $A,B$ consist of positive functionals, decompose $T$ into self-adjoint parts as $T_1 + iT_2$.
Then $\inf_{a\in A} \langle a, T_1\rangle > \sup_{b\in B} \langle b, T_1\rangle$.
\end{Lemma}
\begin{proof}
See any text on functional analysis, for example \cite[Theorems 3.4 and 3.10]{rudin}.
\end{proof}

\begin{Lemma}\label{Lemma P is Dense in M}
$\Pcal_1$ is dense in $\Mcal$.
\end{Lemma}
\begin{proof}
Suppose to the contrary that a mean $m$ lies outside $\cl(\Pcal_1)$.
Applying Lemma~\ref{Lemma Hahn Banach} to the sets $\{m\}$ and $\cl(\Pcal_1)$, obtain self-adjoint $T\in X$ such that
$\langle m, T\rangle > \sup_{u\in\Pcal_1} \langle u, T\rangle$.
Letting $S = T + \|T\| I \geq 0$, we have $\langle m, S\rangle > \sup_{u\in\Pcal_1} \langle u, S\rangle = \|S\|$,
contradicting $\|m\| = 1$.
\end{proof}
%%%%%%%%%%%%%%%%%%%%%
%%%%%%%%%%%%%%%%%%%%%
%%%%%%%%%%%%%%%%%%%%%
\chapter{Topologically Invariant Means on \texorpdfstring{$\Linfty$}{L-infinity(G)}}
\section{TI-nets in \texorpdfstring{$\Pcal_1(G)$}{P1(G)}}
\begin{point}
In previous sections, $|E|$ has denoted the cardinal number of $E$.
But when $E\subset G$, $|E|$ is understood to denote its left Haar measure.
Integrals are always taken with respect to left Haar measure.
The modular function $\triangle: G\to \Rbb^\times$ is a continuous homomorphism defined by $|Ut| = |U| \triangle(t)$.
The map $f\mapsto f^\dag$ defined by $f^\dag(x) = f(x^{-1}) \triangle(x^{-1})$ is an involution of $\Lone$ that sends $\Pcal_1(G)$ to itself.
Left and right translation are defined by $l_t f(x) = f(t^{-1} x)$ and $r_t f(x) = f(xt)$, so that $l_{xy} = l_x l_y$ and $r_{xy} = r_x r_y$.
Additionally, we define $R_t f(x) = f(xt^{-1}) \triangle(t^{-1})$, so that $R_t \frac{1_{U}}{|U|} = \frac{1_{U t}}{|U t|}$.
Notice $R_{xy} = R_y R_x$.
\end{point}

\begin{point}\label{Definition_fT}
For $T\in\Linfty$ and $f,g\in\Lone$, define $fT$ by $\langle fT, g\rangle = \langle T, f*g\rangle$,
and $T f$ by $\langle T f, g\rangle = \langle T, g*f\rangle$.
Equivalently, $fT(t) = \langle R_t f, T\rangle$, and $Tf(t) = \langle l_t f, T\rangle$.

A mean $m$ is said to be topologically left invariant if $\langle m, fT\rangle  = \langle m, T\rangle$ for all $f\in\Pcal_1(G)$ and $T\in\Linfty$.
(This is equivalent to the traditional definition, $\langle m, f*T\rangle = \langle m, T\rangle$.)
Likewise, $m$ is topologically right invariant if $\langle m,Tf\rangle = \langle m, T\rangle$, and topologically two-sided invariant if it is both of the above.
The sets of topologically left/right/two-sided invariant means on $\Linfty$ are denoted $\TLIM(G) / \TRIM(G) / \TIM(G)$.
\end{point}

\begin{point}
A net $\{f_\gamma\}\subset\Pcal_1(G)$ is called a weak left TI-net if $\langle g*f_\gamma - f_\gamma, T\rangle \to 0$ for all $g\in\Pcal_1(G)$ and $T\in\Linfty$.
By Lemma~\ref{Lemma P is Dense in M}, topologically left invariant means are precisely the limit points of weak left TI-nets.
$\{f_\gamma\}$ is simply called a left TI-net if $\|g*f_\gamma - f_\gamma\|_1 \to 0$ for all $g\in\Pcal_1(G)$.
The method of \cite{WeakNotStrong} can be used to construct weak TI-nets that are not TI-nets, which makes Corollary~\ref{Cor TI-nets suffices} interesting.

The analogous definitions and statements for right/two-sided TI-nets are obvious.
\end{point}

\begin{point}
As in \ref{Intro_Number_of_Tims}, suppose $\Kcal$ is a collection of compact sets covering $G$ with $|\Kcal| = \kappa$.
Let $\Kcal'$ be the set of all finite unions in $\Kcal$.
Pick any compact $U\subset G$ with nonempty interior.
Then $\Kcal'' = \{UK : K\in \Kcal'\}$ is a collection of compact sets satisfying:
(1) $|\Kcal''| = \kappa$.
(2) $\Kcal''$ is closed under finite unions.
(3) $\bigcup_{K\in\Kcal''}K^\circ = G$.
\\Hence there is no loss of generality in supposing $\Kcal$ itself satisfies (1)-(3).
It follows that every compact $C\subset G$ is contained in some $K\in\Kcal$.
\end{point}

\begin{point}\label{Definition_Folner_Net}
Suppose $G$ is noncompact amenable.
Let $\Gamma = \{(K, n) : K\in\Kcal,\ n\in\Nbb\}$.
F\o{}lner's condition says that for each $\gamma = (K,n)$, we can pick a compact set $F_\gamma$ that is $(K, \frac1n)$-left-invariant.
In other words, letting $\lambda_\gamma = \frac{1_{F_\gamma}}{|F_\gamma|}$,
 we have
	$\| l_x \lambda_\gamma - \lambda_\gamma\|_1 < \frac1n$ for each $x\in K$.
Now $\Gamma$ becomes a directed set when endowed with the following partial-order:
$[(K,m) \preceq (J,n)] \iff [K\subseteq J$ and $m\leq n]$.
By condition (2) above, notice that each tail of $\Gamma$ has cardinality $\kappa \cdot \Nbb = \kappa = |\Gamma|$, as required by Lemma~\ref{Lemma Cardinality of Gamma*}.
\end{point}

\begin{Lemma}\label{Lemma lambda_gamma is left TI-net}
$\{\lambda_\gamma\}$ is a left TI-net.
\end{Lemma}
\begin{proof}
Pick $f\in\Pcal_1(G)$.
Since compactly supported functions are dense in $\Pcal_1(G)$, we may suppose $f$ has compact support $C$.
Pick $K\in\Kcal$ containing $C$.
\\Notice $\|f*\lambda_\gamma - \lambda_\gamma\|_1
	\leq \int \int f(t) |\lambda_\gamma(t^{-1}x) - \lambda_\gamma(x)| \dd{x} \dd{t}
	= \int f(t) \|l_t \lambda_\gamma - \lambda_\gamma\|_1 \dd{t}
	\leq \sup_{t\in C} \|l_t \lambda_\gamma - \lambda_\gamma\|_1$.
\\In particular, $\|f*\lambda_\gamma - \lambda_\gamma\|_1 < 1/n$ whenever $\gamma \succeq (K, n)$.
\end{proof}

\begin{Corollary}\label{Cor TI-nets suffices}
Every topologically left invariant mean is the limit of a left TI-net.
(Likewise for right/two-sided invariant means, although we omit the proof for those cases.)
\end{Corollary}
\begin{proof}
Suppose $m$ is not the limit of any left TI-net.
Then there exist $T_1, \hdots, T_n \in\Linfty$, $ f\in\Pcal_1(G)$, and $\epsilon > 0$
such that for any $g\in\Pcal_1(G)$ with $\|f*g - g\|_1 < \epsilon$, $\max_{i=1\hdots n}|\langle m-g, T_i\rangle| > \epsilon$.
\\Let $\gamma$ be large enough that $\|f * \lambda_\gamma - \lambda_\gamma\|_1 < \epsilon$, and let $\lambda = \lambda_\gamma$.
\\By Lemma~\ref{Lemma P is Dense in M}, pick $g\in\Pcal_1(G)$ with $\max_{i=1\hdots n}|\langle m-g, T_i\rangle| < \epsilon$ and $\max_{i=1\hdots n} |\langle m-g, \lambda T_i\rangle| < \epsilon$.
\\Now $\|f*(\lambda*g) - (\lambda * g)\|_1 < \epsilon$,
	so by hypothesis $|\langle m-\lambda*g, T_i\rangle| > \epsilon$ for some $i$.
\\But by choice of $g$, $|\langle m-\lambda*g, T_i\rangle|
	= |\langle m, T_i\rangle - \langle g, \lambda T_i\rangle|
	= |\langle m - g, \lambda T_i\rangle|
	< \epsilon$, a contradiction.
\end{proof}

\begin{Corollary}
Let $\rho_\gamma = \lambda_\gamma^\dag$. Then $\{\rho_\gamma\}$ is a right TI-net.
\end{Corollary}
\begin{proof}
For any $f\in \Pcal_1(G)$, $\|\rho_\gamma *f - \rho_\gamma\|_1 = \|f^\dag*\lambda_\gamma - \lambda_\gamma\|_1\to 0$.
\end{proof}

\begin{Lemma}\label{Lemma_Product_of_TI_Nets}
If $\{f_\gamma\}$ is a left TI-net, and $\{g_\gamma\}$ is a right TI-net, then $\{f_\gamma *g_\gamma\}$ is a two-sided TI-net.
\end{Lemma}
\begin{proof}
First of all, $f_\gamma* g_\gamma \geq 0$ and $\|f_\gamma* g_\gamma\|_1 = 1$, so $\{f_\gamma *g_\gamma\} \subset \Pcal_1(G)$.
Now for each $h\in\Pcal_1(G)$, $\|h*f_\gamma* g_\gamma - f_\gamma* g_\gamma\| \leq \|h*f_\gamma - f_\gamma\| \cdot \|g_\gamma\| \to 0$
and $\|f_\gamma* g_\gamma* h - f_\gamma* g_\gamma\| \leq \|f_\gamma\| \cdot \|g_\gamma *h - g_\gamma\|\to 0$.
\end{proof}

\begin{point}
The next lemma generalizes \cite[Theorem 3.2]{Chou70}.
Intuitively it says, ``we can prove facts about the entire set $\TLIM(G)$, simply by proving them about the right-translates of a single left TI-net.''
After that, we have generalizations to $\TRIM(G)$ and $\TIM(G)$.
\end{point}

\begin{Lemma}\label{Lemma_cl_conv_Xp}
For any $p\in\Gamma^*$,
let $X_p = \{\plim_\gamma [R_{t_\gamma} \lambda_\gamma] : \{t_\gamma\}\in G^\Gamma\}.$
Then $\cl(\conv(X_p)) = \TLIM(G)$.
\end{Lemma}
\begin{proof}
Suppose $m \in\TLIM(G)$ lies outside the closed convex hull of $X_p$.
\\As in the proof of Lemma~\ref{Lemma P is Dense in M},
	there exists $T\in L_\infty(G,\Rbb)$ with $\langle m,T\rangle > \sup_{\nu\in X_p}\langle\nu,T\rangle$.
\\For each $\gamma$,
$\langle m,T\rangle
	= \langle m, \lambda_\gamma T\rangle
	\leq \|\lambda_\gamma T\|_\infty
	= \sup_t [\lambda_\gamma T(t)]
	= \sup_t \langle R_t \lambda_\gamma,\ T \rangle$.
\\In particular, if $\gamma = (K,n)$, choose $t_\gamma$ so that
$\langle m, T\rangle < \langle R_{t_\gamma} \lambda_\gamma, T\rangle + 1/n$.
\\Now $\langle m,T\rangle \leq \langle\plim_\gamma [R_{t_\gamma}\lambda_\gamma],T\rangle$, contradicting our choice of $m$.
\end{proof}

\begin{Lemma}
For any $p\in\Gamma^*$,
let $X_p = \{\plim_\gamma [l_{t_\gamma} \rho_\gamma] : \{t_\gamma\}\in G^\Gamma\}.$
Then $\cl(\conv(X_p)) = \TRIM(G)$.
\end{Lemma}
\begin{proof}
Essentially the same as the proof of Lemma~\ref{Lemma_cl_conv_Xp} but the third line becomes:
\\For each $\gamma$,
$\langle m,T\rangle
	= \langle m, T\rho_\gamma\rangle
	\leq \|T \rho_\gamma\|_\infty
	= \sup_t [T\rho_\gamma(t)]
	= \sup_t \langle l_t \rho_\gamma,\ T \rangle$.
\end{proof}

\begin{Lemma}
For any $p\in\Gamma^*$, 
let $X_p = \{\plim_\gamma [\lambda_\gamma *(l_{t_\gamma} \rho_\gamma)]: \{t_\gamma\}\in G^\Gamma\}$.
Then $\cl(\conv(X_p)) = \TIM(G)$.
\end{Lemma}
\begin{proof}
Essentially the same as the proof of Lemma~\ref{Lemma_cl_conv_Xp} but the third line becomes:
\\For each $\gamma$,
	$\langle m,T\rangle
	= \langle m, \lambda_\gamma T \rho_\gamma\rangle
	\leq \|\lambda_\gamma T \rho_\gamma\|_\infty
	=\sup_t[\lambda_\gamma T \rho_\gamma(t)]
	=\sup_t \langle \lambda_\gamma * (l_t \rho_\gamma),\ T\rangle$.
\end{proof}

\begin{Lemma} \label{Lemma Disjoint Folner}
There exists $\{t_\gamma\}\in G^\Gamma$ such that $\{F_\gamma t_\gamma F_\gamma^{-1}\}$ are mutually disjoint.
Since $\supp(\lambda_\gamma *(l_t \rho_\gamma)) \subset F_\gamma t F_\gamma^{-1}$, it follows from \ref{Linfty is operators} and Lemma~\ref{Lemma_Product_of_TI_Nets} that $\{\lambda_\gamma *(l_t \rho_\gamma)\}$ is an orthogonal TI-net.
\end{Lemma}
\begin{proof}
Since $|\Gamma| = \kappa$, let $(\Gamma, <)$ denote the well-ordering of $\Gamma$ induced by some bijection with $\kappa$.
Let $t_0 = e$.
As induction hypothesis, suppose $\{F_\gamma t_\gamma F_\gamma^{-1}\}_{\gamma < \alpha}$ are disjoint, where $\alpha\in\Gamma$.
If it is impossible to find $t_\alpha$ such that $\{F_\gamma t_\gamma F_\gamma^{-1}\}_{\gamma \leq \alpha}$ are disjoint, then
$\{F_\alpha^{-1} F_\gamma t_\gamma F_\gamma^{-1} F_\alpha\}_{\gamma < \alpha}$ is a collection of compact sets of cardinality less than $\kappa$ covering $G$, a contradiction.
\end{proof}

\begin{Lemma}\label{L-UpperBoundSizeOfTlim}
Let $K\subset G$ be any compact set with nonempty interior, and $\{K t_\gamma\}_{\gamma < \kappa}$ be a covering of $G$ by translates of $K$.
For each $n$, let $\lambda_n\in\Pcal_1(G)$ be $(K^{-1}, \frac1n)$-left-invariant, and let $\rho_n = \lambda_n^\dag$.
\\Finally, let $X = \left\{ R_{t_\gamma}\rho_n : n\in\Nbb,\ \gamma < \kappa\right\}$. 
Then $\TLIM(G) \subset \cl(\conv(X))$, and $|\cl(\conv(X))| \leq 2^{2^\kappa}$.
\end{Lemma}
\begin{proof}
Suppose $m\in \TLIM(G)$ lies outside the closed convex hull of $X$.
As in Lemma~\ref{Lemma P is Dense in M}, there exists $T\in L_\infty(G,\Rbb)$ and $\epsilon > 0$
	with $\langle m, T\rangle > 2\epsilon +  \langle R_{t_\gamma} \rho_n, T\rangle$ for all $n\in\Nbb$ and all $\gamma < \kappa$.
\\Let $n$ be large enough that $\|T\|_\infty / n < \epsilon$.
\\Let $s$ be chosen so that $\sup_t \langle R_t \rho_n, T\rangle < \epsilon + \langle R_{s} \rho_n, T\rangle$. Say $s\in K t_{\alpha}$.
\\By definition of $\rho_n$,
	$\langle R_{s} \rho_n, T\rangle <
  \langle R_{t_\alpha} \rho_n, T\rangle + \|T\|_\infty/n$.
\\Now $\langle m, T\rangle
	= \langle m, \rho_n T\rangle
	\leq \|\rho_n T\|_\infty
	= \sup_t \langle R_t \rho_n, T\rangle
	< 2\epsilon + \langle R_{t_\alpha} \rho_n, T\rangle$ contradicting the choice of $m$.
\\This proves $\TLIM(G) \subset \cl(\conv(X))$.
\\Let $\conv_\Qbb(X)$ denote the set of all finite convex combinations in $X$ with rational coefficients.
Evidently $|\conv_\Qbb(X)| = |X| = \kappa$, and $\cl(\conv(X)) = \cl(\conv_\Qbb(X))$.
Since $\cl(\conv(X))$ is a regular Hausdorff space with dense subset $\conv_\Qbb(X)$, it has cardinality at most $2^{2^{|\conv_\Qbb(X)|}} = 2^{2^\kappa}$.
\end{proof}

\begin{Theorem}
$|\TIM(G)| = |\TLIM(G)| = 2^{2^\kappa}$.
\end{Theorem}
\begin{proof}
By Lemma~\ref{Lemma Disjoint Folner}, there exists be an orthogonal TI-net $\{\varphi_\gamma\}$.
By Lemma~\ref{Lemma Orthogonal Implies Injection},
the map $p\mapsto \plim_\gamma \varphi_\gamma$ is one-to-one from $\Gamma^*$ to $\TIM(G)$, so $|\TIM(G)| \geq |\Gamma^*|$.
By Lemma~\ref{Lemma Cardinality of Gamma*}, $|\Gamma^*| = 2^{2^\kappa}$.
Of course, $\TIM(G) \subset \TLIM(G)$, so $2^{2^\kappa} \leq |\TIM(G)| \leq |\TLIM(G)|$.
Lemma~\ref{L-UpperBoundSizeOfTlim} yields the opposite inequality.
\end{proof}
%%%%%%%%%%%%%%%%%%%%%
%%%%%%%%%%%%%%%%%%%%%
%%%%%%%%%%%%%%%%%%%%%
\section{A Proof of Paterson's Conjecture} \label{S-PatersonsConjecture}
Let $t^g$ be shorthand for $g t g^{-1}$.
We write $G\in[\FC]^-$ to signify that each conjugacy class $t^G = \{t^g : g\in G\}$ is precompact.
When $G\in [FC]^-$ is furthermore discrete, each conjugacy class must be finite.
In this case, we write $G\in \FC$.

\begin{Lemma} \label{T-FCisAmenable}
If $G\in \FC$, then $G$ is amenable.
\end{Lemma}
\begin{proof}
It suffices to show that each finitely-generated subgroup is amenable.
Suppose $K\subset G$ is finite, and let $\langle K \rangle$ denote the subgroup generated by $K$.
For any $x\in \langle K\rangle$, let $C(x) = \{y \in \langle K\rangle : x^y = x\}$.
Evidently $C(x)$ is a subgroup, whose right cosets correspond to the (finitely many) distinct values of $x^y$.
Therefore $[\langle K\rangle : C(x)]$ is finite.
Let $Z$ denote the center of $\langle K\rangle$.
Clearly $Z = \bigcap_{x \in K} C(x)$.
Thus $[\langle K\rangle: Z] \leq \prod_{x\in K} [\langle K\rangle : C(x)]$ is finite.
Since $\langle K\rangle$ is finite-by-abelian, it is amenable.
\end{proof}

\begin{Theorem} \label{T-Structure}
$G\in[\FC]^-$ iff $G$ is a compact extension of $\Rbb^n \times D$, for some $D\in\FC$ and $n\in\Nbb$.
\end{Theorem}
\begin{proof}
This is \cite[Theorem 2.2]{liukkonen}.
\end{proof}

\begin{Corollary}
If $G\in [\FC]^-$, then $G$ is amenable.
\end{Corollary}
\begin{proof}
In light of Theorem~\ref{T-FCisAmenable}, any group of the form $\Rbb^n \times D$ is amenable.
\\Compact extensions of amenable groups are amenable.
\end{proof}

\begin{Corollary} \label{C-FCisUnimodular}
If $G\in [\FC]^-$, then $G$ is unimodular.
\end{Corollary}
\begin{proof}
Clearly groups of the form $\Rbb^n \times D$ are unimodular, since $D$ is discrete.
\\Compact extensions of unimodular groups are unimodular.
\end{proof}

\begin{Lemma}\label{Lemma_CG_is_precompact}
If $G\in[\FC]^-$, and $C\subset G$ is compact, then $C^G = \{ c^g : c\in C, g\in G\}$ is precompact.

In \cite{milnes}, Milnes obverves that this lemma would imply Theorem~\ref{Theorem If FC Bar then...}.
He is unable to prove it, apparently because he is unaware of Theorem~\ref{T-Structure}.
\end{Lemma}
\begin{proof}
By Theorem~\ref{T-Structure}, let $G/K = \Rbb^n \times D$, where $K\normal G$ is a compact normal subgroup.
\\Let $\pi: G\to G/K$ denote the canonical epimorphism.
\\Pick $C\subset G$ compact.
Notice $C^G$ is precompact if $\pi(C^G)$ is, because the kernel of $\pi$ is compact.
\\Since $\pi(C)$ is compact, $\pi(C) \subset B \times F$ for some box $B\subset\Rbb^n$ and finite $F\subset D$.
\\Now $\pi(C^G) = \pi(C)^{\pi(G)} \subset (B\times F)^{\Rbb^n \times D} = B\times F^D$.
\\Evidently $F^D$ is finite, hence $B\times F^D$ is compact, proving $\pi(C^G)$ is precompact.
\end{proof}
\begin{point}\label{Symmetric_Folner_Net}
As in \ref{Definition_Folner_Net}, let $\{F_\gamma\}_{\gamma\in\Gamma}$ be a F\o{}lner net for $G$, and $\lambda_\gamma = 1_{F_\gamma} / |F_\gamma|$ the corresponding left TI-net.
Assuming $G$ is unimodular, by \cite[Theorem 4.4]{Chou70} we can choose $F_\gamma$ to be symmetric.
Hence $\lambda_\gamma(x) = \lambda_\gamma(x^{-1}) = \lambda_\gamma^\dag(x)$, and $\{\lambda_\gamma\}$ is a two-sided TI-net.
\end{point}

\begin{Theorem}[Following Milnes]\label{Theorem If FC Bar then...}
If $G\in[\FC]^-$ then $\TLIM(G) \subset \TIM(G)$.
\end{Theorem}
\begin{proof}
By Corollary~\ref{C-FCisUnimodular}, $G$ is unimodular.
As above, take $\{\lambda_\gamma\} = \{\lambda_\gamma^\dag\}$ to be a TI-net.
\\Recall Lemma~\ref{Lemma_cl_conv_Xp}, which says $\cl(\conv(X_p)) = \TLIM(G)$.
So it suffices to prove $X_p \subset \TIM(G)$.
\\To this end, we will show $\{R_{t_\gamma} \lambda_\gamma\}$ is a right TI-net for any $\{t_\gamma\}\in G^\Gamma$.
Let $x_\gamma$ be shorthand for $t_\gamma x t_\gamma^{-1}$.
\\Now $\|r_x R_{t_\gamma} \lambda_\gamma - R_{t_\gamma} \lambda_\gamma \|_1
	= \frac{|F_\gamma t_\gamma x^{-1} \Delta F_\gamma t_\gamma|}{|F_\gamma t_\gamma|}
	= \frac{|F_\gamma x_\gamma^{-1} \Delta F_\gamma|}{|F_\gamma|}
	= \| r_{x_\gamma} \lambda_\gamma - \lambda_\gamma \|_1
	= \| l_{x_\gamma} \lambda_\gamma - \lambda_\gamma \|_1$.
\\For any compact $C\subset G$, notice $\{x_\gamma : x\in C, \gamma\in\Gamma\} \subset C^G$, which is precompact by Lemma~\ref{Lemma_CG_is_precompact}.
\\Hence $\sup_{x\in C} \|r_x R_{t_\gamma} \lambda_\gamma - R_{t_\gamma} \lambda_\gamma \|_1
	\leq \sup_{y\in C^G} \|l_y \lambda_\gamma - \lambda_\gamma\|\to 0$.
\\By the same argument as Lemma~\ref{Lemma lambda_gamma is left TI-net},
$\|(R_{t_\gamma} \lambda_\gamma)* f- R_{t_\gamma} \lambda_\gamma\|_1 \to 0$ for any $f\in\Pcal_1(G)$.
\end{proof}

\begin{Lemma}
If $x^G$ is not precompact, then there is a $\sigma$-compact open subgroup $H\leq G$ such that $x^H$ is not precompact.
\end{Lemma}
\begin{proof}
Let $U$ be any compact neigborhood of $t_0 = e$.
Inductively pick $t_{n+1}$ such that $x^{t_{n+1}}U \cap \{x^{t_0}, \hdots, x^{t_n}\}U = \varnothing$.
This is possible, otherwise $x^G$ is contained in the compact set $\{x^{t_0}, \hdots, x^{t_n}\} U U^{-1}$, a contradiction.
Now the sequence $\{x^{t_0}, x^{t_1}, \hdots\}$ does not accumulate anywhere, so it escapes any compact set.
Take $H$ to be the subgroup generated by $\{x^{t_0}, x^{t_1}, \hdots\} U$.
\end{proof}

\begin{Theorem}\label{Theorem Paterson Converse}
If $G$ has an element $x$ such that $x^G$ is not precompact, then $\TLIM(G) \neq \TIM(G)$.
\end{Theorem}
\begin{proof}
Let $H \leq G$ be any $\sigma$-compact open subgroup such that $x^H$ is not precompact.
Pick a sequence of compact sets $\{H_n\}$ such that
$\forall n$ $H_n \subset H_{n+1}^\circ$,
and $\bigcup_n H_n = H$.
It follows that any compact $K\subset H$ is contained in some $H_n$.

Inductively construct a sequence $\{c_n\}\subset H$ satisfying the following properties:
\\(1) $c_n \in H\setminus K_n$, where $K_n = \bigcup_{m < n} H_n^{-1} H_m c_m \{x, x^{-1}\}$.
\\(2) $(c_n x c_n^{-1}) \not\in H_n^{-1} H_n$.
\\Since $K_n$ is compact, $x^{H \setminus K_n}$ is not precompact, and in particular escapes $H_n^{-1} H_n$.
\\Thus it is possible to satisfy (1) and (2) simultaneously.

Now $A = \bigcup_n H_n c_n$ and $B = \bigcup_n H_n c_n x$ are easily seen to be disjoint:
\\If they are not disjoint, then $H_n c_n x \cap H_m c_m \neq \varnothing$ for some $n,m$.
\\If $n > m$, then $c_n \in H_n^{-1} H_m c_m x^{-1}$, violating (1).
\\If $n < m$, then $c_m \in H_m^{-1} H_n c_n x$, violating (1).
\\If $n=m$, then $c_n x c_n^{-1} \in H_n^{-1} H_n$, violating (2).

Let $T$ be a transversal for $G/H$.
Notice $TA \cap TB = \varnothing$, since $A,B$ are disjoint subsets of $H$.
\\Define $\pi: TH \to H$ by $\pi(th) = h$, which is continuous since $H$ is open.
\\Let $\{F_\gamma\}$ be a F\o{}lner net for $G$.
Each $F_\gamma$ is compact, hence $\pi[F_\gamma] \subset H_{n(\gamma)}$ for some $n(\gamma)$.
Let $c_\gamma = c_{n(\gamma)}$.
Now $\pi[F_\gamma c_\gamma] = \pi[F_\gamma]c_\gamma \subset A$, hence $F_\gamma c_\gamma \subset \pi^{-1}[A] = TA$.
Likewise $F_\gamma c_\gamma x \subset TB$.
\\Since this holds for each $\gamma$, $C = \bigcup_\gamma F_\gamma t_\gamma \subset TA$ and $D = \bigcup_\gamma F_\gamma t_\gamma x \subset TB$.
Thus $C \cap D = \varnothing$.
\\Let $m = \plim_\gamma\left[R_{t_\gamma} \frac{1_{F_\gamma}}{|F_\gamma|}\right]$ for some $p\in\Gamma^*$.
\\Now $m(1_C) = 1$, but $m(r_x 1_C) = m(1_D) = 0$, hence $m\in \TLIM(G) \setminus \TIM(G)$.
\end{proof}
%%%%%%%%%%%%%%%%%%%%%
%%%%%%%%%%%%%%%%%%%%%
%%%%%%%%%%%%%%%%%%%%%
\chapter{Topologically Invariant means on \texorpdfstring{$VN(G)$}{VN(G)}}
\section{Background on Fourier Algebra}
The algebras $\AG$ and $\VN$ are defined and studied in the influential paper \cite{Eymard}.
Chapter 2 of \cite{Lau-Kaniuth} gives a modern treatment of the same material in English.

\begin{point}
Let $\lambda$ denote the left-regular representation.
By definition, $\VN = \lambda[G]'' = \lambda[\Lone]''$.
$\Pcal_1(\Ghat)$ denotes the set of normal states on $\VN$, and the linear span of $\Pcal_1(\Ghat)$ is denoted by $\AG$.
It turns out that $\Pcal_1(\Ghat)$ consists of all vector states of the form $\omega_f(T) = \langle Tf,f\rangle$ where $\|f\|_2 = 1$.
By the polarization identity, $\AG$ consists of all functionals of the form $\omega_{f,g}(T) = \langle Tf,g\rangle$, with $f,g\in\Ltwo$.
\end{point}

\begin{point}
$\AG$ can also be characterized as the set of all functions $u_{f,g}(x) = \langle \lambda(x) f, g\rangle = \overline{g} * \check{f}(x)$ with $f,g\in\Ltwo$.
As such, the elements of $\AG$ are continuous and vanish at infinity.
The duality with $\VN$ is given by $\langle T, u_{f,g}\rangle = \langle Tf,g\rangle$.
In particular, $\langle \lambda(h), u\rangle = \int h(x) u(x)\dd{x}$ for all $h\in\Lone$.

$\AG$ is a commutative Banach algebra with pointwise operations and norm $\|u\| = \sup_{\|T\| = 1} |\langle T, u\rangle|$.
Equivalently, $\|u\|$ is the sup of $\left|\int h(x) u(x)\dd{x}\right|$ over all $h\in\Lone$ such that $\|\lambda(h)\|\leq 1$.
Since $\|\lambda(h)\|\leq \|h\|_1$, we conclude $\|u\| \geq \|u\|_\infty$.
\end{point}

\begin{point}
If $u\in\AG$ is positive as a functional on $\VN$, we write $u\geq 0$.
It's easy to see $u\geq 0$
iff $u = u_{f,f}$ for some $f\in\Ltwo$
iff $\|u\|_\infty = u(e) = \langle I, u\rangle = \|u\|$.
In particular, $\Pcal_1(\Ghat) = \{u\in\AG : \|u\| = u(e) = 1\}$ is closed under multiplication.
Suppose $\|f\|_2 = \|g\|_2 = 1$.
Then $\|u_{f,f} - u_{g,g}\|
	= \sup_{\|T\|=1} |\langle Tf, f\rangle - \langle Tg, g\rangle|
	\leq \sup_{\|T\|=1} |\langle T(f-g), f\rangle| + |\langle Tg, (f-g)\rangle| \leq 2\|f-g\|_2$.
\end{point}
%%%%%%%%%%%%%%%%%%%%%%%%%
%%%%%%%%%%%%%%%%%%%%%%%%%
%%%%%%%%%%%%%%%%%%%%%%%%%
\section{TI-nets in \texorpdfstring{$\Pcal_1(\widehat{G})$}{P1(G\^{})}}
\begin{point}
For $T\in\VN$ and $u,v\in\AG$, define $uT$ by $\langle uT, v\rangle = \langle T, uv\rangle$.
A mean $m$ is said to be topologically invariant if
$\langle m, uT\rangle  = \langle m, T\rangle$ for all $u\in\Pcal_1(\Ghat)$ and $T\in\VN$.
$\TIM(\Ghat)$ denotes the set of topologically invariant means on $\VN$.
There is no distinction between left/right topologically invariant means, since multiplication in $\Pcal_1(\Ghat)$ is commutative.
\end{point}

\begin{point}
A net $\{u_\gamma\}\subset\Pcal_1(\Ghat)$ is called a weak TI-net if
$\lim_\gamma \langle vu_\gamma - u_\gamma, T\rangle = 0$ for all $v\in\Pcal_1(\Ghat)$ and $T\in\VN$.
By Lemma~\ref{Lemma P is Dense in M}, topologically invariant means are precisely the limit points of weak TI-nets.
$\{u_\gamma\}$ is simply called a TI-net if $\lim_\gamma \|vu_\gamma - u_\gamma\| = 0$ for all $v\in\Pcal_1(\Ghat)$.
\end{point}

\begin{Lemma}\label{Lemma Small Supp Means TI-net}
Pick any net $\{u_\gamma\}\subset \Pcal_1(\Ghat)$.
Suppose that $\supp(u_\gamma)$ is eventually small enough to fit in any $V\in \Cscr(G)$.
Then $\{u_\gamma\}$ is a TI-net.
\end{Lemma}
\begin{proof}
This is \cite[Proposition 3]{renaud}.
I'll repeat the proof for convenience.
\\Pick $u\in\Pcal_1(\Ghat)$.
Since compactly supported functions are dense in $\AG$, it suffices to consider the case when $u$ has compact support $C$.
$\AG$ is regular by \cite[Proposition 2.3.2]{Lau-Kaniuth}, so we can pick $v\in\AG$ with $v\equiv 1$ on $C$.
Since $u-v(e) = 0$, by \cite[Lemma 2.3.7]{Lau-Kaniuth} we can find some $w\in\AG$ with $\|(u-v) - w\| < \epsilon$ and $w\equiv 0$ on some neighborhood $W$ of $e$.
Suppose $\gamma$ is large enough that $\supp(u_\gamma) \subset W\cap C$.
Then $v u_\gamma = u_\gamma$, and $wu_\gamma = 0$,
hence $\|u u_\gamma - u_\gamma\|
	= \|(u-v) u_\gamma\|
	\leq \|(u-v-w)u_\gamma\| + \|w u_\gamma\|
	\leq \|(u-v-w)\| < \epsilon$.
\end{proof}

\begin{point}
Let $\Cscr(G)$ denote the compact, symmetric neighborhoods of $e$.
As a consequence of Lemma~\ref{Lemma Small Supp Means TI-net}, we see that $\Pcal_1(\Ghat)$ has TI-nets for any $G$, in contrast to $\Pcal_1(G)$ which has TI-nets only when $G$ is amenable:
If $\{U_\gamma\}\subset \Cscr(G)$ is a neighborhood basis at $e$, directed by inclusion, let $u_\gamma = (1_{U_\gamma} * 1_{U_\gamma}) / |U_\gamma|$.
Since $\supp(u_\gamma) \subset U_\gamma^2$, and $\{U_\gamma^2\}$ is a neighborhood basis at $e$, $\{u_\gamma\}$ is a TI-net.

Lemma~\ref{Lemma Small Supp Means TI-net} has a sort of converse, given by the following lemma.
\end{point}

\begin{Lemma}\label{Lemma Every TIM(Ghat) is limit of TI-net}
Every $m\in \TIM(\Ghat)$ is the limit of a net $\{u_\gamma\}$
such that $\supp(u_\gamma)$ is eventually small enough to fit in any $V\in\Cscr(G)$.
\end{Lemma}
\begin{proof}
Suppose to the contrary that there exist $V\in\Cscr(G)$, $T\in\VN$, and $\epsilon > 0$ such that $|\langle m-u, T\rangle| > \epsilon$ for all $u\in\Pcal_1(\Ghat)$ with $\supp(u) \subset V$.
Fix any $u\in\Pcal_1(\Ghat)$ with $\supp(u) \subset V$.
By Lemma~\ref{Lemma P is Dense in M}, pick $v\in\Pcal_1(\Ghat)$ with $|\langle m - v, T\rangle | < \epsilon$,
and $|\langle m - v, uT\rangle| < \epsilon$.
Now $vu\in \Pcal(\Ghat)$ with $\supp(vu)\subset V$.
But $|\langle m - vu, T\rangle|
	=|\langle m, T\rangle - \langle vu, T\rangle|
	= |\langle m, uT\rangle - \langle v, uT\rangle|
	= |\langle m - v, uT\rangle| < \epsilon$, a contradiction.
\end{proof}

\begin{Corollary}
Every $m\in\TIM(\Ghat)$ is the limit of a TI-net.
\end{Corollary}

\begin{Corollary}\label{Corollary Upper Bound on TIM Ghat}
$|\TIM(\Ghat)| \leq 2^{2^\mu}$
\end{Corollary}
\begin{proof}
Recall that $\mu$ is the minimal cardinality of a neighborhood basis at $e$.
Pick any $V\in\Cscr(G)$. By the previous lemma, $\TIM(\Ghat)$ is in the closure of $X = \{u\in \Pcal_1(\Ghat) : \supp(u) \subset V\}$.
Therefore, it suffices to find a set $D$ with $|D| = \mu$ and $X \subset \cl(D)$.

Let $B = \{f\in\Ltwo : \|f\|_2 = 1,\ \supp(f) \subset V\}$, and let $C$ comprise the continuous functions in $B$.
Obviously $C$ is dense in $B$, and it is routine to show that $C$ has a dense subset $C'$ of cardinality $\mu$.
Now $D = \{u_{f,f} : f\in C'\}$ is the desired set.
\end{proof}

\begin{point}\label{rank one support}
Let $u_{f,f}$ be a normal state on $\VN$, and $P$ the projection onto the linear span of $f$.
Since $P$ has rank one, it is clearly the inf of all projections $Q$ such that $1 = \langle Q f, f\rangle = \langle Q, u_{f,f}\rangle$.
In the terminology of \ref{Definition of Support}, $P$ is the support of $u_{f,f}$.
Thus a family $\{u_{f_\alpha, f_\alpha}\} \subset \AG$ is orthogonal iff $\{f_\alpha\} \subset \Ltwo$ is orthogonal.

Apparently neither Chou nor Hu made this observation.
If it seems obvious, it is because we took care to express normal states as $u_{f,f}$.
For example, in order to construct an orthogonal TI-sequence, \cite[pages 210-213]{Chou82} takes three pages to describe a generalized Gram-Schmidt procedure for normal states on a von Neumann algebra.
This procedure does not generalize beyond sequences, and \cite{Hu95} is unable to construct an orthogonal net when $\mu > \Nbb$.
\end{point}

\begin{Lemma}\label{Lemma_Orthogonal_TI_Sequence}
When $\mu = \Nbb$, $\Pcal_1(\Ghat)$ contains an orthogonal TI-sequence.
\end{Lemma}
\begin{proof}
Let $\{U_n\} \subset \Cscr(G)$ be a descending neighborhood basis at $e$, such that $|U_n \setminus U_{n+1}| > 0$.
Let $V_n = U_n \setminus U_{n+1}$, and $f_n = 1_{V_n} / |V_n|^{1/2}$.
In light of Lemma~\ref{Lemma Small Supp Means TI-net} and \ref{rank one support}, $\{u_{f_n, f_n}\}$ is the desired sequence.
\end{proof}

\begin{Lemma}\label{Lemma_Minimal_Cardinality_Subbasis}
Let $\Ucal \subset \Cscr(G)$. If $\bigcap \Ucal = \{e\}$, then $\Ucal$ is a neighborhood sub-basis at $e$ (hence $|\Ucal| \geq \mu$).
\end{Lemma}
\begin{proof}
Suppose $\bigcap \Ucal = \{e\}$, and let $V$ be any open neighborhood of $e$.
Since $\varnothing = \{e\}\cap V^C = \bigcap_{U\in \Ucal} [U\cap V^C]$ is an intersection of compact sets, it follows that some finite sub-intersection is also empty.
In other words, $U_1\cap \hdots \cap U_n \subset V$ for some $U_1, \hdots, U_n \in \Ucal$.
\end{proof}

\begin{Lemma}[Kakutani-Kodaira]\label{Lemma KK}
Let $\Ucal \subset \Cscr(G)$.
Suppose each $U \in \Ucal$ has a ``successor'' $V\in \Ucal$, such that $V^2 \subset U$.
Then $H = \bigcap \Ucal$ is a compact subgroup of $G$.
\end{Lemma}
\begin{proof}
$H$ is compact, because it is the intersection of compact sets in a Hausdorff space.
\\Suppose $x,y \in H$. Suppose $U, V \in \Ucal$ with $V^2 \subset U$.
Since $x,y \in H\subset V$, $xy \in V^2 \subset U$.
\\Since $U$ was arbitrary, we conclude $xy\in H$.
Likewise, $x^{-1}\in H$ because each $U$ is symmetric.
\end{proof}

\begin{point}
Suppose $H$ is a compact subgroup of $G$, with normalized Haar measure $\nu$.
Let $L^2(H \backslash G)$ be the set of functions in $L^2(G)$ that are constant on right-cosets of $H$.
Define $Pf(x) = \int f(hx) \dd\nu(h)$.
Now it is routine to check that $P$ is the orthogonal projection onto $L^2(H \backslash G)$.
The only interesting detail, needed to prove $P = P^*$, is that $\int f(hx) \dd\nu(h) = \int f(h^{-1}x) \dd\nu(h)$ because $H$ is unimodular.
\end{point}

\begin{point}\label{KK Small U and H}
Suppose $\{H_\gamma\} = \{H_\gamma\}_{\gamma < \mu}$ is a descending chain of compact subgroups.
Let $P_\gamma$ denote the orthogonal projection onto $L^2(H_\gamma \backslash G)$.
Then $\{P_\gamma\}$ is an ascending chain of projections, and $\{P_{\gamma + 1} - P_\gamma\}$ is a chain of mutually orthogonal projections.
Suppose we construct functions $\{f_\gamma\} \subset \Ltwo$ with $f_\gamma \in L^2(H_{\gamma + 1} \backslash G) - L^2(H_\gamma \backslash G)$, so that $(P_{\gamma+1} - P_\gamma) f_\gamma \neq 0$.
Letting $g_\gamma = (P_{\gamma+1} - P_\gamma) f_\gamma$, we see $\{u_\gamma\}= \{u_{g_\gamma, g_\gamma} / \|g_\gamma\|_2^2\}$ is a chain of mutually orthogonal functions in $\Pcal_1(\Ghat)$, since $S(u_\gamma) \leq P_{\gamma+1} - P_\gamma$.

Consider the condition $f_\gamma \in L^2(H_{\gamma+1} \backslash G) - L^2(H_\gamma \backslash G)$.
How can we achieve this?
Let $\nu_\gamma$ denote the Haar measure of $H_\gamma$.
Pick $U_\gamma \in \Cscr(G)$ small enough that $\nu_\gamma(U_\gamma^4 \cap H_\gamma) < 1$.
Construct $H_{\gamma + 1}$ small enough that $H_{\gamma + 1} \subset U_\gamma$, and let $f_\gamma = 1_{H_{\gamma + 1} U_\gamma}$. Obviously $f_\gamma \in L^2(H_{\gamma+1} \backslash G)$.
On the other hand,
$P_\gamma f_\gamma(x) = \nu_\gamma(\{h\in H_\alpha : hx \in H_{\gamma + 1} U_\gamma\}) < 1 = f_\gamma(x)$
for any $x\in H_{\gamma + 1} U_\gamma$.
Hence $P_\gamma f_\gamma \neq f_\gamma$, and $f_\gamma\not\in L^2(H_\gamma \backslash G)$.

We need to add one more detail to our construction, to make the chain of orthogonal functions $\{u_\gamma\}$ into a TI-net.
Namely, suppose $\{V_\gamma\}\subset\Cscr(G)$ is a neighborhood basis at $e$, and $\{U_\gamma, H_\gamma\}$ satisfy $H_\gamma U_\gamma^2 \subset V_\gamma$.
Then $\supp(g_\gamma) \subset H_\gamma H_{\gamma+1} U_\gamma \subset V_\gamma$,
hence $\supp(u_\gamma) \subset V_\gamma^2$.
Let $\Gamma$ denote the set of ordinals less than $\mu$, ordered by $\alpha\prec\gamma \iff V_\gamma \subset V_\alpha$.
By Lemma~\ref{Lemma Small Supp Means TI-net}, $\{u_\gamma\}_{\gamma\in\Gamma}$ is a TI-net.
Notice that each tail of $\Gamma$ has cardinality $|\Gamma| = \mu$, as required by Lemma~\ref{Lemma Cardinality of Gamma*}.
\end{point}

\begin{Lemma}\label{Lemma_Orthogonal_TI_Net}
When $\mu > \Nbb$, $\Pcal_1(\Ghat)$ contains an orthogonal TI-net of cardinality $\mu$.
\end{Lemma}
\begin{proof}
Fix $\{V_\gamma\}_{\gamma < \mu} \subset\Cscr(G)$, a well-ordered neighborhood basis at $e$.
Of course this well-ordering has no topological meaning, but it's necessary for transfinite induction.
The purpose of our induction is to select $\{U_\gamma, H_\gamma\}_{\gamma < \mu}$ as in \ref{KK Small U and H}, from which the lemma clearly follows.

For $0 \leq \beta < \mu$, suppose we have picked $\{U_\gamma, H_\gamma\}_{\gamma < \beta}$ such that (1)-(4) hold for all $\gamma < \beta$:
\\(1) $H_\gamma$ is the intersection of $(\gamma+1)\cdot \Nbb$ elements of $\Cscr(G)$, and is a subgroup of each previous $H_\alpha$.
\\(2) If $\gamma=\alpha + 1$ is a successor ordinal, then $H_\gamma \subset U_\alpha$.
\\(3) $H_\gamma U_\gamma^2 \subset V_\gamma$.
\\(4) $\nu_\gamma(H_\gamma \cap U_\gamma^4) < 1$.
\\Let $\{W_n\}\subset \Cscr(G)$ be any chain with $W_0 \subset V_\beta$ and $W_{n+1}^2 \subset W_n$ for all $n$.
If $\beta = \alpha + 1$, we may suppose $W_0 \subset U_\alpha$.
Let $H_\beta = \bigcap (\{W_n\} \cup \{H_\gamma\}_{\gamma < \beta})$.
Notice $H_\beta W_2^2 \subset W_0 \subset V_\beta$.
Since $H_\beta$ is the intersection of $(\beta+1) \cdot \Nbb < \mu$ elements of $\Cscr(G)$,
it follows from Lemma~\ref{Lemma_Minimal_Cardinality_Subbasis} that $H_\beta \neq \{e\}$.
In particular, it is possible to pick $U\in \Cscr(G)$ with $\nu_\beta (H_\beta \cap U^4) < 1$.
Let $U_\beta = U \cap W_2$.
Now $U_\beta, H_\beta$ satisfy (1)-(4).
\end{proof}

\begin{Theorem}
$|\TIM(\Ghat)| = 2^{2^\mu}$.
\end{Theorem}
\begin{proof}
Let $\{u_\gamma\}_{\gamma\in\Gamma}$ be the orthogonal TI-net of of Lemma~\ref{Lemma_Orthogonal_TI_Sequence} or Lemma~\ref{Lemma_Orthogonal_TI_Net}, depending on $\mu$.
\\By Lemma~\ref{Lemma Orthogonal Implies Injection}, $p\mapsto \plim_\gamma u_\gamma$ is an injection of $\Gamma^*$ into $\TIM(\Ghat)$.
\\Thus $|\TIM(\Ghat)| \geq |\Gamma^*| = 2^{2^\mu}$ by Lemma~\ref{Lemma Cardinality of Gamma*}.
The opposite inequality is Corollary~\ref{Corollary Upper Bound on TIM Ghat}.
\end{proof}
%%%%%%%%%%%%%%%%%%%%%
%%%%%%%%%%%%%%%%%%%%%
%%%%%%%%%%%%%%%%%%%%%
\bibliographystyle{alpha}
\bibliography{enumeratingTim}
\end{document}